\documentclass[10pt,a4paper]{amsart}

\theoremstyle{plain}
\newtheorem{theorem}{Theorem}[section]
\newtheorem{lemma}[theorem]{Lemma}
\newtheorem{proposition}[theorem]{Proposition}
\newtheorem{corollary}[theorem]{Corollary}

\theoremstyle{definition}

\theoremstyle{remark}
\newtheorem{remark}[theorem]{Remark}

\newtheorem{conjecture}{Conjecture}

\newtheorem{notations}[theorem]{Notations}

\newcommand{\ic}{\ensuremath{\mathcal{I}}}
\newcommand{\oc}{\ensuremath{\mathcal{O}}}

\newcommand{\fc}{\ensuremath{\mathcal{F}}}
\newcommand{\dc}{\ensuremath{\mathcal{D}}}

\newcommand{\hc}{\ensuremath{\mathcal{H}}}

\newcommand{\mc}{\ensuremath{\mathcal{M}}}

\newcommand{\Ps}{\mathbb{P}}

\newcommand{\bC}{\mathbb{C}}

\newcommand{\bG}{\mathbb{G}}

\def\bin #1#2 {\left( \matrix { #1 \cr #2 \cr } \right) }

\newcommand{\tG}{\tilde{G}}

\newcommand{\tX}{\tilde{X}}

\begin{document}

\title[Monodromy of a family of hypersurfaces]
{Monodromy of a family of hypersurfaces}

\author{Vincenzo Di Gennaro }
\address{Universit\`a di Roma \lq\lq Tor Vergata\rq\rq, Dipartimento di Matematica,
Via della Ricerca Scientifica, 00133 Roma, Italy.}
\email{digennar@axp.mat.uniroma2.it}

\author{Davide Franco }
\address{Universit\`a di Napoli
\lq\lq Federico II\rq\rq, Dipartimento di Matematica e
Applicazioni \lq\lq R. Caccioppoli\rq\rq, P.le Tecchio 80, 80125
Napoli, Italy.} \email{davide.franco@unina.it}

\bigskip
\abstract Let $Y$ be an $(m+1)$-dimensional irreducible smooth
complex projective variety embedded in a projective space. Let $Z$
be  a closed subscheme of $Y$, and $\delta$ be a positive integer
such that $\mathcal I_{Z,Y}(\delta)$ is generated by global
sections. Fix an integer $d\geq \delta +1$, and assume the general
divisor $X \in |H^0(Y,\ic_{Z,Y}(d))|$ is smooth. Denote by
$H^m(X;\mathbb Q)_{\perp Z}^{\text{van}}$  the quotient of
$H^m(X;\mathbb Q)$ by the cohomology of $Y$ and also by the cycle
classes of the irreducible components of dimension $m$ of $Z$. In
the present paper we prove that the monodromy representation on
$H^m(X;\mathbb Q)_{\perp Z}^{\text{van}}$ for the family of smooth
divisors $X \in |H^0(Y,\ic_{Z,Y}(d))|$ is irreducible.

\bigskip\noindent
{{R\'ESUM\'E.}}$\,$ Soit $Y$ une vari\'et\'e projective complexe
lisse irr\'eductible de dimension $m+1$, plong\'ee dans un espace
projectif. Soit $Z$ un sous-sch\'ema ferm\'e de $Y$, et soit
$\delta$ un entier positif tel que $\mathcal I_{Z,Y}(\delta)$ soit
engendr\'e par ses sections globales. Fixons un entier $d\geq
\delta +1$, et supposons que le diviseur g\'en\'eral $X \in
|H^0(Y,\ic_{Z,Y}(d))|$ soit lisse. D\'esignons par $H^m(X;\mathbb
Q)_{\perp Z}^{\text{van}}$ le quotient de $H^m(X;\mathbb Q)$ par
la cohomologie de $Y$ et par les classes des composantes
irr\'eductibles de $Z$ de dimension $m$. Dans cet article nous
prouvons que la repr\'esentation de monodromie sur $H^m(X;\mathbb
Q)_{\perp Z}^{\text{van}}$ pour la famille des diviseurs lisses $X
\in |H^0(Y,\ic_{Z,Y}(d))|$ est irr\'eductible.



\endabstract

\maketitle

\section{Introduction}

In this paper we  provide an affirmative answer to a question
formulated in \cite{OS}.

Let $Y\subseteq \Ps^N$ ($dim\,Y=m+1$) be an irreducible smooth
complex projective variety embedded in a projective space $\Ps^N$,
$Z$ be a closed subscheme of $Y$, and  $\delta$ be a positive
integer such that $\mathcal I_{Z,Y}(\delta)$ is generated by
global sections. Assume that for $d\gg 0$ the general divisor $X
\in |H^0(Y,\ic_{Z,Y}(d))|$ is smooth. In the paper \cite{OS} it is
proved that this is equivalent to the fact that the strata
$Z_{\{j\}}=\{x\in Z\,:\,dim\,T_xZ=j\}$, where $T_xZ$ denotes the
Zariski tangent space,  satisfy the following inequality:
\begin{equation}
\label{Kleiman} dim\,Z_{\{j\}}+j\leq dim\,Y-1 \quad{\text{for
any}}\quad j\leq dim\,Y.
\end{equation}

\noindent This property implies that, for any $d\geq \delta$,
there exists a smooth hypersurface of degree $d$ which contains
$Z$ (\cite{OS}, 1.2. Theorem).

It is generally expected that, for $d\gg 0$, the Hodge cycles of
the general hypersurface $X$ $\in$ $|H^0(Y,\ic_{Z,Y}(d))|$ depend
only on $Z$ and on the ambient variety $Y$. A very precise
conjecture in this direction was made in \cite{OS}:

\begin{conjecture}[Otwinowska - Saito]
\label{conj} Assume $deg\,X\geq \delta +1$. Then the monodromy
representation on $H^m(X;\mathbb Q)_{\perp Z}^{\text{van}}$ for
the family of smooth divisors $X\in |H^0(Y,\mathcal O_Y(d))|$
containing $Z$ as above is irreducible.
\end{conjecture}
\noindent We denote by
 $H^m(X;\mathbb Q)_Z^{\text{van}}$  the subspace of
$H^m(X;\mathbb Q)^{\text{van}}$ generated by the cycle classes of
the maximal dimensional irreducible components of $Z$ modulo the
image of $H^m(Y;\mathbb Q)$ (using the orthogonal decomposition
$H^m(X;\mathbb Q)=H^m(Y;\mathbb Q)\perp H^m(X;\mathbb
Q)^{\text{van}}$) if $m=2\,dim\,Z$, and $H^m(X;\mathbb
Q)_Z^{\text{van}}=0$ otherwise, and we denote by $H^m(X;\mathbb
Q)_{\perp Z}^{\text{van}}$  the orthogonal complement of
$H^m(X;\mathbb Q)_Z^{\text{van}}$ in $H^m(X;\mathbb
Q)^{\text{van}}$. The conjecture above cannot be strengthened
because, even in $Y=\mathbb{P}^3$, there exist examples for which
$dim\,H^m(X;\mathbb Q)_{\perp Z}^{\text{van}}$ is arbitrarily
large and the monodromy representation associated to the linear
system $|H^0(Y, \ic_{Z,Y}(\delta))|$  is diagonalizable.

The Authors of \cite{OS} observed that a proof for such a
conjecture  would confirm the expectation above and would reduce
the Hodge conjecture for the general hypersurface $X_t \in
|H^0(Y,\ic_{Z,Y}(d))|$ to the Hodge conjecture for $Y$. More
precisely, by a standard argument, from Conjecture \ref{conj} it
follows that when $m=2\,dim\,Z$ and the vanishing cohomology of
the general $X_t \in |H^0(Y,\ic_{Z,Y}(d))|$ ($d\geq \delta +1$) is
not of pure Hodge type $(m/2,m/2)$, then the Hodge cycles in the
middle cohomology of $X_t$ are generated by the image of the Hodge
cycles on $Y$ together with the cycle classes of the irreducible
components of $Z$. So, the Hodge conjecture for $X_t$ is reduced
to that for $Y$ (compare with \cite{OS}, Corollary 0.5). They also
proved that the conjecture is satisfied in the range $d\geq\delta
+2$, or for $d=\delta +1$ if hyperplane sections of $Y$ have non
trivial top degree holomorphic forms (\cite{OS}, 0.4. Theorem).
Their proof relies on Deligne's semisimplicity Theorem and on
Steenbrink's Theory for semistable degenerations.

Arguing in a  different way, we prove in this paper Conjecture
\ref{conj} in full.  More precisely, avoiding degeneration
arguments, in Section $2$ we will deduce Conjecture \ref{conj}
from the following:

\begin{theorem}
\label{thm2} Fix integers $1\leq k<d$, and let $W=G\cap X\subset
Y$ be a complete intersection of smooth divisors $G\in
|H^0(Y,\mathcal O_Y(k))|$ and $X\in |H^0(Y,\mathcal O_Y(d))|$.
Then the monodromy representation on $H^m(X;\mathbb Q)_{\perp
W}^{\text{van}}$ for the family of smooth divisors $X_t\in
|H^0(Y,\mathcal O_Y(d))|$ containing $W$ is irreducible.
\end{theorem}

\noindent Here we define $H^m(X;\mathbb Q)_{\perp W}^{\text{van}}$
in a similar way as before, i.e. as the orthogonal complement in
$H^m(X;\mathbb Q)^{\text{van}}$ of the image $H^m(X;\mathbb Q)_{
W}^{\text{van}}$ of the  map obtained by composing the natural
maps $H_m(W;\mathbb Q)\to H_m(X;\mathbb Q)\cong H^m(X;\mathbb
Q)\to H^m(X;\mathbb Q)^{\text{van}}$.

The proof of Theorem  \ref{thm2} will be given in Section $4$ and
consists in a Lefschetz type argument applied to the image of the
rational map on $Y$ associated to the linear system
$|H^0(Y,\ic_{W,Y}(d))|$, which turns out to have at worst isolated
singularities. This approach was started in our paper \cite{DGF}
where we proved a particular case of Theorem \ref{thm2}, but the
proof given here is independent and much  simpler.

We begin by proving   Conjecture \ref{conj} as a consequence of
Theorem \ref{thm2}, and next we prove Theorem \ref{thm2}.

\section{Proof of Conjecture \ref{conj} as a consequence of Theorem \ref{thm2}.}

We keep the same notation we introduced before, and need further
preliminaries.

\begin{notations}\label{notazioni}
(i) Let $V_{\delta}\subseteq H^0(Y,\mathcal I_{Z,Y}(\delta))$ be a
subspace generating $\mathcal I_{Z,Y}(\delta)$, and
$V_{d}\subseteq H^0(Y,\mathcal I_{Z,Y}(d))$ ($d\geq \delta +1$) be
a subspace containing the image of $V_{\delta}\otimes H^0(\Ps^N,
\mathcal O_{\Ps^N}(d-\delta))$ in $H^0(Y,\mathcal I_{Z,Y}(d))$.
Let $G\in |V_{\delta}|$ and $X\in |V_d|$ be divisors. Put
$W:=G\cap X$. From condition (\ref{Kleiman}), and \cite{OS}, 1.2.
Theorem, we know that if $G$ and $X$ are general then they are
smooth. Moreover, by (\cite{Vogel}, p. 133, Proposition 4.2.6. and
proof), we know that if $G$ and $X$ are smooth then $W$ has only
isolated singularities.

(ii) In the case $m>2$, fix a smooth $G\in |V_{\delta}|$. Let
$H\in |H^0(\Ps^N, \mathcal O_{\Ps^N}(l))|$ be a general
hypersurface of degree $l\gg 0$, and put $Z':= Z\cap H$ and
$G':=G\cap H$. Denote by  $V'_{d}\subseteq H^0(G',\mathcal
I_{Z',G'}(d))$ the restriction of   $V_{d}$ on $G'$, and by
$V''_{d}\subseteq H^0(G,\mathcal I_{Z,G}(d))$ the restriction of
$V_{d}$ on $G$. Since $H^0(G,\mathcal I_{Z,G}(d))\subseteq
H^0(G',\mathcal I_{Z',G'}(d))$,  we may identify $V''_{d}=V'_{d}$.
Put $W':=W\cap H\in |V'_{d}|$. Similarly as we did for the triple
$(Y,X,Z)$, using the orthogonal decomposition $H^{m-2}(W';\mathbb
Q)=H^{m-2}(G';\mathbb Q)\perp H^{m-2}(W';\mathbb Q)^{\text{van}}$,
we  define the subspaces $H^{m-2}(W';\mathbb Q)_{Z'}^{\text{van}}$
and $H^{m-2}(W';\mathbb Q)_{\perp Z'}^{\text{van}}$ of
$H^{m-2}(W';\mathbb Q)$ with respect to the triple $(G',W',Z')$.
Passing from $(Y,X,Z)$ to $(G',W',Z')$  will allow us to prove
Conjecture \ref{conj} arguing by induction on $m$ (see the proof
of Proposition \ref{induzione} below).

(iii) Let $\varphi:\mathcal W \to |V''_{d}|$ ($\mathcal W
\subseteq G \times |V''_{d}|$) be the universal family
parametrizing the divisors  $W=G\cap X\in |V''_{d}|$. Denote by
$\sigma: \widetilde { \mathcal W}\to \mathcal W$ a
desingularization of $\mathcal W$, and by $U_{\varphi}\subseteq
|V''_{d}|$ a nonempty open set such that the restriction
$(\varphi\circ \sigma)_{|U_{\varphi}}: (\varphi\circ
\sigma)^{-1}{(U_{\varphi})}\to U_{\varphi}$ is smooth. Next, let
$\psi:\mathcal W' \to |V'_{d}|$ ($\mathcal W' \subseteq G \times
|V'_{d}|$) be the universal family parametrizing the divisors
$W'=W\cap H\in |V'_{d}|$, and denote by $U_{\psi}\subseteq
|V'_{d}|$ a nonempty open set such that the restriction
$\psi_{|U_{\psi}}: \psi^{-1}{(U_{\psi})}\to U_{\psi}$ is smooth.
Shrinking $U_{\varphi}$ and $U_{\psi}$ if necessary, we may assume
$U:= U_{\varphi}=U_{\psi}\subseteq |V''_{d}|=|V'_{d}|$. For any
$t\in U$ put $W_t:=\varphi^{-1}(t)$, ${\widetilde
W}_t:={\sigma}^{-1}(W_t)$, and $W'_t:=\psi^{-1}(t)$. Observe that
$W_t\cap Sing(\mathcal W)\subseteq Sing(W_t)$, so we may assume
$W'_t=W_t\cap H\subseteq W_t\backslash Sing(W_t)\subseteq
{\widetilde W}_t$. Denote by $\iota_t$ and $\tilde \iota_t$ the
inclusion maps $W'_t\to W_t$ and $W'_t\to {\widetilde W}_t$. The
pull-back maps ${\tilde \iota_t}^*: H^{m-2}({\widetilde
W}_t;\mathbb Q)\to H^{m-2}(W'_t;\mathbb Q)$ give rise to a natural
map ${\tilde \iota}^*: R^{m-2}((\varphi\circ
\sigma)_{|U})_{*}\mathbb Q\to R^{m-2}({\psi_{|U}})_{*}\mathbb Q $
between local systems on $U$, showing that $Im({\tilde
\iota_t}^*)$ is globally invariant under the monodromy action on
the cohomology of the smooth  fibers  of $\psi$. Finally, we
recall that the inclusion map $\iota_t$ defines a Gysin map
$\iota_t^{\star}: H_{m}(W_t;\mathbb Q)\to H_{m-2}(W'_t;\mathbb Q)$
(see \cite{Fulton}, p. 382, Example 19.2.1).
\end{notations}

\begin{remark}\label{note}
Fix a smooth $G\in |V_{\delta}|$, and assume $m\geq 2$. The linear
system $|V_d|$ induces an embedding of $G\backslash Z$ in some
projective space: denote by $\Gamma $ the image of $G\backslash Z$
through this embedding. Since $G\backslash Z$ is irreducible, then
also $\Gamma$ is, and so is its general hyperplane section, which
is isomorphic to $(G\cap X)\backslash Z$ via $|V_d|$. So we see
that, when $m\geq 2$, for any smooth $G\in |V_{\delta}|$ and any
general $X\in |V_d|$, one has that $W\backslash Z$ is irreducible.
In particular, when $m>2$, then also $W$ is irreducible.
\end{remark}

\begin{lemma}
\label{morfismo} Fix a smooth $G\in |V_{\delta}|$, and assume
$m>2$. Then, for  a general $t\in U$,  one has
$Im({\tilde\iota_t}^*)=Im(PD\circ \,\iota_t^{\star})$, and the map
$PD\circ \,\iota_t^{\star}$ is injective (PD means {\lq\lq
Poincar\'e duality\rq\rq}: $H_{m-2}(W'_t;\mathbb Q)\cong
H^{m-2}(W'_t;\mathbb Q)$).
\end{lemma}
\begin{proof}
By (\cite{Voisin}, p. 385, Proposition 16.23) we know that
$Im({\tilde\iota_t}^*)$ is equal to the image of the pull-back
$H^{m-2}(W_t\backslash Sing(W_t);\mathbb Q)\to
H^{m-2}(W'_t;\mathbb Q)$. On the other hand, by  (\cite{Dimca2},
p. 157 Proposition 5.4.4., and p. 158 (PD)) we have natural
isomorphisms involving intersection cohomology groups:
\begin{equation}\label{IntCohom}
H^{m-2}(W_t\backslash Sing(W_t);\mathbb Q)\cong IH^{m-2}(W_t)
\end{equation}
$$
\cong IH^{m}(W_t)^{\vee}\cong H^{m}(W_t;\mathbb Q)^{\vee}\cong
H_{m}(W_t;\mathbb Q).
$$
So we may identify the pull-back $H^{m-2}(W_t\backslash
Sing(W_t);\mathbb Q)\to H^{m-2}(W'_t;\mathbb Q)$ with $PD\circ
\,\iota_t^{\star}$. This proves that
$Im({\tilde\iota_t}^*)=Im(PD\circ \,\iota_t^{\star})$. Moreover,
since $W'_t$ is smooth, then $IH^{m-2}(W'_t)\cong
H^{m-2}(W'_t;\mathbb Q)$ (\cite{Dimca2}, p. 157). So, from
(\ref{IntCohom}), we may identify $PD\circ \,\iota_t^{\star}$ with
the natural map $IH^{m-2}(W_t)\to IH^{m-2}(W_t\cap H)$, which  is
injective in view of Lefschetz Hyperplane Theorem for intersection
cohomology (\cite{Dimca2}, p. 158 (I), and p. 159, Theorem 5.4.6)
(recall that $W'_t=W_t\cap H$).
\end{proof}

We are in position to prove Conjecture \ref{conj}.

Fix a smooth $G\in |V_{\delta}|$, and a general $X\in |V_d|$. Put
$W=G\cap X$. Since the monodromy group of the family of smooth
divisors $X\in |H^0(Y,\mathcal O_Y(d))|$ containing $W$ is a
subgroup of the monodromy group of the family of smooth divisors
$X\in |H^0(Y,\mathcal O_Y(d))|$ containing $Z$, in order to deduce
Conjecture \ref{conj} from Theorem \ref{thm2}, it suffices to
prove that  $H^m(X;\mathbb Q)_{\perp Z}^{\text{van}}=H^m(X;\mathbb
Q)_{\perp W}^{\text{van}}$. Equivalently, it suffices to prove
that  $H^m(X;\mathbb Q)_{Z}^{\text{van}}=H^m(X;\mathbb Q)_{
W}^{\text{van}}$. This is the content of the following:

\begin{proposition}
\label{induzione} For any smooth $G\in |V_{\delta}|$ and any
general $X\in |V_d|$, one has $H^m(X;\mathbb Q)_{
Z}^{\text{van}}=H^m(X;\mathbb Q)_{ W}^{\text{van}}$.
\end{proposition}
\begin{proof}
First we analyze the cases $m=1$ and $m=2$, and next we argue by
induction on $m>2$ (recall that $dim\,Y=m+1$).

The case $m=1$ is trivial because in this case $dim\,Z\leq dim
\,W=0$.

Next assume $m=2$. In this case $dim \,Y=3$ and $dim \,Z\leq 1$.
Denote by $Z_1,\dots,Z_h$ ($h\geq 0$) the irreducible components
of $Z$ of dimension $1$ (if there are). Fix a smooth $G\in
|V_{\delta}|$ and a general $X\in |V_d|$, and put $W=G\cap
X=Z_1\cup\dots\cup Z_h\cup C$, where $C$ is the residual curve,
with respect to $Z_1\cup\dots\cup Z_h$, in the complete
intersection $W$. By Remark \ref{note} we know that $C$ is
irreducible. Then, as (co)cycle classes, $Z_1,\dots,Z_h, C$
generate $H^2(X;\mathbb Q)_{ W}^{\text{van}}$, and $Z_1,\dots,Z_h$
generate $H^2(X;\mathbb Q)_{ Z}^{\text{van}}$. Since
$Z_1+\dots+Z_h+ C=\delta H_X$ in $H^2(X;\mathbb Q)$ ($H_X$=
general hyperplane section of $X$ in $\Ps^N$), and this cycle
comes from $H^2(Y;\mathbb Q)$, then $Z_1+\dots+Z_h+ C=0$ in
$H^2(X;\mathbb Q)^{\text{van}}$, and so $H^2(X;\mathbb Q)_{
Z}^{\text{van}}=H^2(X;\mathbb Q)_{ W}^{\text{van}}$. This
concludes the proof of Proposition \ref{induzione} in the case
$m=2$.

Now assume $m>2$ and argue by induction on $m$. First we observe
that the intersection pairing on $H^{m-2}(W';\mathbb Q)_{
Z'}^{\text{van}}$ is non-degenerate: this follows from Hodge Index
Theorem, because the cycles in $H^{m-2}(W';\mathbb Q)_{
Z'}^{\text{van}}$ are primitive and algebraic. So we have the
following orthogonal decomposition:
\begin{equation}\label{orthogonaldecomp}
H^{m-2}(W';\mathbb Q)=H^{m-2}(G';\mathbb Q)\perp
H^{m-2}(W';\mathbb Q)_{ Z'}^{\text{van}}\perp H^{m-2}(W';\mathbb
Q)_{\perp Z'}^{\text{van}}.
\end{equation}
Let $\mathcal J$ be the local system on $U$ with fibre given by
$H^{m-2}(G';\mathbb Q)$ $\perp$ $H^{m-2}(W';\mathbb
Q)_{Z'}^{\text{van}}$. We claim that:
\begin{equation}\label{gei}
Im(\tilde\iota ^*)=\mathcal J.
\end{equation}

We will prove (\ref{gei}) shortly after. From (\ref{gei}) and
Lemma \ref{morfismo} we get an isomorphism: $ H_{m}(W;\mathbb
Q)\cong H^{m-2}(G';\mathbb Q)\perp H^{m-2}(W';\mathbb Q)_{
Z'}^{\text{van}}. $ Taking into account that by Lefschetz
Hyperplane Theorem we have $H^{m-2}(Y;\mathbb Q)\cong
H^{m-2}(G;\mathbb Q)\cong H^{m-2}(G';\mathbb Q)$, and that the
Gysin map $H_{m}(Z;\mathbb Q)$ $\to$ $H_{m-2}(Z';\mathbb Q)$ is
bijective (because $H_{m}(Z;\mathbb Q)$ and $H_{m-2}(Z';\mathbb
Q)$ are simply generated by the components which are of dimension
$m$ or $m-2$ of $Z$ and $Z'$ (if there are)), one sees that the
natural map $H_m(W;\mathbb Q)\to H_m(X;\mathbb Q)\cong
H^m(X;\mathbb Q)$ sends $H^{m-2}(G';\mathbb Q)$ in
$H^{m}(Y;\mathbb Q)$, and $H^{m-2}(W';\mathbb Q)_{
Z'}^{\text{van}}$ in $H^{m}(X;\mathbb Q)_{ Z}^{\text{van}}$. This
proves $H^m(X;\mathbb Q)_{ Z}^{\text{van}}$ $\supseteq$
$H^m(X;\mathbb Q)_{ W}^{\text{van}}$. Since the reverse inclusion
is obvious, it follows that $H^m(X;\mathbb Q)_{ Z}^{\text{van}}$
$= H^m(X;\mathbb Q)_{ W}^{\text{van}}$.

So, to conclude the proof of Proposition \ref{induzione}, it
remains to prove claim (\ref{gei}). To this purpose first notice
that $Im({\tilde\iota_t}^*)$ contains $H^{m-2}(W'_t;\mathbb
Q)_{Z'}^{\text{van}}$, because, by Lemma \ref{morfismo}, we have
$Im({\tilde\iota_t}^*)=Im(PD\circ \,\iota_t^{\star})$, and
$Im(PD\circ \,\iota_t^{\star})\supseteq H^{m-2}(W'_t;\mathbb
Q)_{Z'}^{\text{van}}$ in view of the quoted isomorphism
$H_{m}(Z;\mathbb Q)\cong H_{m-2}(Z';\mathbb Q)$. Moreover
$Im({\tilde\iota_t}^*)$ contains $H^{m-2}(G';\mathbb Q)$ because
$H^{m-2}(G';\mathbb Q)\cong H^{m-2}(G;\mathbb Q)$, and
$H^{m-2}(G;\mathbb Q)$ is contained in $Im({\tilde\iota_t}^*)$.
Therefore we obtain $Im({\tilde\iota}^*)\supseteq \mathcal J$,
from which we deduce that $Im({\tilde\iota}^*)= \mathcal J$. In
fact, otherwise, since by induction $H^{m-2}(W'_t;\mathbb
Q)_{\perp Z'}^{\text{van}}$ is irreducible, from
(\ref{orthogonaldecomp}) it would follow that
$Im({\tilde\iota}^*)=R^{m-2}({\psi_{|U}})_{*}\mathbb Q$. This is
impossible because for $l\gg 0$ the dimension of
$H^{m-2}(W'_t;\mathbb Q)$ is arbitrarily large (by the way, we
notice that the same argument proves that $\mathcal J$ is nothing
but the invariant part of $R^{m-2}({\psi_{|U}})_{*}\mathbb Q$).
\end{proof}

\section{A Monodromy Theorem}

In this section we  prove a monodromy theorem (see Theorem
\ref{ngeneralfact} below), which we will use in next section for
proving Theorem \ref{thm2}, and that we think of independent
interest.

Let $Q\subseteq\Ps$ be an irreducible, reduced, non-degenerate
projective variety of dimension $m+1$ ($m\geq 0$), with isolated
singular points $q_1,\dots,q_r$. Let $L\in \bG(1,\Ps^*)$ be a
general pencil of hyperplane sections of $Q$, and denote by $Q_L$
the blowing-up of $Q$ along the base locus of $L$, and by $f:
Q_L\to L$ the natural map. The ramification locus of $f$ is a
finite set $\{q_1,\dots,q_s\}:=Sing(Q)\cup\{q_{r+1},\dots,q_s\}$,
where $\{ q_{r+1}, \dots ,q_s \}$ denotes the set of tangencies of
the pencil. Set $a_i:=f(q_i)$, $1\leq i\leq s$ (compare with
\cite{Steenbrink}, p. 304). The restriction map $f: Q_L \backslash
f^{-1}(\{ a_1, \dots ,a_s \}) \to L\backslash\{ a_1, \dots ,a_s
\}$ is a smooth proper map. Hence the fundamental group
$\pi_1(L\backslash\{ a_1, \dots ,a_s \}, t)$ ($t=$ general point
of $L$) acts by monodromy on $Q_t:=f^{-1}(t)$, and so on
$H^{m}(Q_t; \mathbb{Q})$. By \cite{PS}, p. 165-167, we know that
$f: Q_L \backslash f^{-1}(\{ a_1, \dots ,a_s \}) \to L\backslash\{
a_1, \dots ,a_s \}$ induces an orthogonal decomposition:
$H^{m}(Q_t; \mathbb{Q})=I\perp V$, where $I$ is the subspace of
the invariant cocycles, and $V$ is its orthogonal complement.

In the case $Q$ is smooth, a classical basic result in Lefschetz
Theory  states that $V$ is generated by {\lq\lq{standard vanishing
cycles}\rq\rq} (i.e. by vanishing cycles corresponding to the
tangencies of the pencil). This implies the irreducibility of $V$
by standard classical reasonings (\cite{La}, \cite{Voisin}). Now
we are going to prove that it holds true also when $Q$ has
isolated singularities. This is the content of the following
Theorem \ref{ngeneralfact}, for which we didn't succeed in finding
an appropriate reference (for a related and somewhat more precise
statement, see Proposition \ref{ntrivial} below).

\begin{theorem}
\label{ngeneralfact} Let $Q\subseteq\Ps$ be an irreducible,
reduced, non-degenerate projective variety of dimension $m+1\geq
1$, with isolated singularities, and $Q_t$ be a general hyperplane
section of $Q$. Let $H^{m}(Q_t; \mathbb{Q})=I\perp V$ be the
orthogonal decomposition given by the monodromy action on the
cohomology of $Q_t$, where $I$ denotes the invariant subspace.
Then $V$ is generated, via monodromy, by standard vanishing
cycles.
\end{theorem}

\begin{remark}\label{notesugenfact} (i) For a particular case of
Theorem \ref{ngeneralfact}, see \cite{Steenbrink}, Theorem (2.2).

(ii) When $Q$ is a curve, i.e. when $m=0$, then Theorem
\ref{ngeneralfact} follows from the well known fact that the
monodromy group is the full symmetric group (see \cite{ACGH}, pg.
111). So we assume from now on that $m\geq 1$.

(iii) When $Q$ is a cone over a degenerate and necessarily smooth
subvariety of $\Ps$, then $f: Q_L\to L$ has only one singular
fiber $f^{-1}(a_1)$ (i.e. $s=1$). In this case
$\pi_1(L\backslash\{ a_1\}, t)$ is trivial. Therefore we have that
$H^{m}(Q_t; \mathbb{Q})=I$, $V=0$, and Theorem \ref{ngeneralfact}
follows.
\end{remark}

Before proving Theorem \ref{ngeneralfact}, we need some
preliminaries. We keep the same notation we introduced before.

\begin{notations}\label{naltrenotazioni}
(i) Let $R_L\to Q_L$ be a desingularization of $Q_L$. The
decomposition $H^{m}(Q_t; \mathbb{Q})=I\perp V$ can be interpreted
via ${R}_L$ as $I=j^*(H^{m} ({R}_L ;\mathbb{Q}))$ and
$V=Ker(H^{m}(Q_t ; \mathbb{Q})\to H^{m+2} ( {R}_L ; \mathbb{Q}))
\cong Ker(H_{m}(Q_t ; \mathbb{Q})\to H_{m}( {R}_L ; \mathbb{Q}))$,
where $j$ denotes the inclusion $Q_t\subset
 R_L$. Using standard arguments (compare with
\cite{Voisin}, p. 325, Corollaire 14.23) one deduces a natural
isomorphism:
\begin{equation}\label{ndecon}
 \quad V\cong Im(H_{m+1}( {R}_L-g^{-1}(t_1),Q_t ;
\mathbb{Q})\to H_{m}(Q_t ;\mathbb{Q})),
\end{equation}
where $g:R_L\to L$ denotes the composition of $R_L\to Q_L$ with
$f: Q_L\to L$, and $t_1\neq t$ another regular value of $g$.

(ii) For any critical value $a_i$ of $L$ fix a closed disk
$\Delta_i\subset L\backslash\{t_1\} \cong \bC$ with center $a_i$
and radius $0<\rho \ll 1$. As in \cite{La}, (5.3.1) and (5.3.2),
one may prove that $H_{m+1}(R_L -g^{-1}(t_1),Q_t ;
\mathbb{Q})\cong{\oplus}_{i=1}^s H_{m+1}
(g^{-1}(\Delta_i),g^{-1}({a_i+\rho}) ; \mathbb{Q})$. By
(\ref{ndecon}) we have:
\begin{equation}\label{ndecomp}
V=V_1+\dots+V_s,
\end{equation}
where we denote by $V_i$ the image in $H^{m}(Q_t ;\mathbb{Q})$
$\cong H_{m} (g^{-1}({a_i+\rho}); \mathbb{Q})$ of each $H_{m+1}
(g^{-1}(\Delta_i), g^{-1}({a_i+\rho});\mathbb{Q})$. When $r+1\leq
i\leq s$, we recognize in $V_i\subseteq H^{m}(Q_t;\mathbb{Q})$ the
subspace generated by the standard vanishing cocycle $\delta_i$
corresponding to a tangent hyperplane section of $Q$ (see
\cite{La}, \cite{Voisin}, \cite{Steenbrink}).

(iii) Consider again the pencil $f: Q_L\to L$, and let $\Ps_L$ be
the blowing up of $\Ps$ along the base locus $B_L$. For any
$i\in\{1,\dots,s\}$, denote by $D_i\subset\Ps_L$ a closed ball
with center $q_i$ and small radius $\epsilon$. Define $M_i
:=Im(H_{m}(f^{-1}(a_i+\rho)\cap D_i; \mathbb{Q}) \to
H_{m}(f^{-1}(a_i+\rho);\mathbb{Q}))$, with $0<\rho \ll \epsilon\ll
1$. Since $H_{m}(f^{-1}(a_i+\rho);\mathbb{Q})\cong H_{m}(Q_t;
\mathbb{Q})\cong H^{m}(Q_t; \mathbb{Q})$, we may regard
$M_i\subseteq H^{m}(Q_t; \mathbb{Q})$.  When $1\leq i\leq r$,
$M_i$ represents the subspace spanned by the cocycles {\lq\lq
{coming}\rq\rq}  from the singularities of $Q$, and lying in the
Milnor fibre $f^{-1}(a_i+\rho)\cap D_i$. When $r+1\leq i\leq s$,
i.e. when $a_i$ corresponds to  a tangent hyperplane section of
$Q$, then $V_i=M_i$. In general we have:
\begin{equation}
\label{nvanishing} V_i\subseteq M_i \quad{\text{\it{for any
\,$i=1,\dots,s$}}}.
\end{equation}
This is a standard fact, that one may prove as in (\cite{Loj},
(7.13) Proposition). For Reader's convenience, we give  the proof
of property (\ref{nvanishing})  in the Appendix, at the end of the
paper.
\end{notations}

Now we are going to prove Theorem \ref{ngeneralfact}

\begin{proof}[Proof of Theorem \ref{ngeneralfact}]
Let $\pi:\fc \to \Ps^*$ ($\fc \subseteq { \Ps}^*\times \Ps$) be
the universal family parametrizing the hyperplane sections of
$Q\subseteq\Ps$, and denote by $\dc\subseteq\Ps^*$ the
discriminant locus of $\pi$, i.e. the set of hyperplanes
$H\in\Ps^*$ such that $Q\cap H$ is singular. At least
set-theoretically, we have $\dc=Q^*\cup\hc_1\cup\dots\cup\hc_r$,
where $Q^*$ denotes the dual variety of $Q$, and $\hc_j$ denotes
the dual hyperplane of $q_j$ (compare with \cite{Steenbrink}, p.
303).

When the codimension of $Q^*$ in $\Ps^*$ is $1$,  denote by $T_t$
the stalk at $t\in {\Ps^*\backslash\dc}$ of the local subsystem of
$R^m(\pi|_{\pi^{-1}({\Ps^*\backslash\dc})})_*\mathbb{Q}$ generated
by the vanishing cocycle at general point of $Q^*$ (compare with
\cite{OS}, p. 373, or \cite{Steenbrink}, p. 306). If the
codimension of $Q^*$ in $\Ps^*$ is $\geq 2$, put $T_t:=\{0\}$. In
order to prove Theorem \ref{ngeneralfact} it suffices to prove
that $V=T$  ($T:=T_t$). By Deligne Complete Reducibility Theorem
(\cite{PS}, p. 167), we may write $H^{m}(Q_t;\mathbb{Q})$
$=W\oplus T$, for a suitable invariant subspace $W$. Now we claim
the following proposition, which we will prove below:
\begin{proposition}
\label{ntrivial} The monodromy representation on  the quotient
local system with stalk $H^{m}(Q_t; \mathbb{Q})/T_t$ at $t\in
{\Ps^*\backslash\dc}$ is trivial.
\end{proposition}

By previous Proposition \ref{ntrivial} it follows that for any $g$
$\in$ $\pi_1(L\backslash\{ a_1, \dots ,a_s \}, t)$ and any $w\in
W$ there exists $\tau\in T$ such that $w^g=w+\tau$. Then
$\tau=w^g-w\in T\cap W=\{0\}$, and so $w^g=w$. Therefore $W$ is
invariant, i.e. $W\subseteq I$, and since $T\subseteq V$ and
$H^{m}(Q_t;\mathbb{Q})=I\oplus V=W\oplus T$, then we have $T=V$.
\end{proof}

It remains to prove Proposition \ref{ntrivial}. To this aim, we
need some preliminaries. We keep the same notation we introduced
before.

Consider again the universal family $\pi:\fc \to \Ps^*$
parametrizing the hyperplane sections of $Q\subseteq\Ps$. We will
denote by $H_x$ the hyperplane parametrized by $x\in\Ps^*$. Fix a
point $q_i\in Sing(Q)$ (hence $i\in\{1,\dots,r\}$). For general
$L$, $q_i$ is not a base point of the pencil defined by $L$, hence
$Q_L\cong Q$ over $q_i$. Combined with the inclusion $Q_L\subseteq
\fc$, we thus have a natural lift of $q_i$ to a point of $\fc$,
still denoted by $q_i$.

\begin{remark}\label{cono}
If $Q^*$ is contained in $\hc_j$ for some $j\in\{1,\dots,r\}$,
then $Q^*$ is degenerate in $\Ps^*$, and so $Q=Q^{**}$ is a cone
in $\Ps$. Therefore, if $Q$ is not a cone, then $Q^*$ is not
contained in $\hc_j$ for any $j\in\{1,\dots,r\}$. In this case,
for a general line $\ell \subseteq\hc_i$, the set $\ell\cap Q^*$
is finite, and  for any $x\in \ell$, $H_x\cap Q$ has an isolated
singularity at $q_i$.
\end{remark}

\begin{notations} (i) Let $\ell \subseteq\hc_i$ be a general line.
For any $u\in\ell\cap Q^*$, denote by $\Delta_{u}^{\circ}$ an open
disk of $\ell$ with center $u$ and small radius. Consider the
compact $K:=\ell\backslash (\bigcup_{u\in{\ell\cap
Q^*}}\Delta_{u}^{\circ})$. In the Appendix below (see Lemma
\ref{nWhitney}) we prove that {\it{there is a closed ball
$D_{q_i}\subseteq { \Ps}^*\times \Ps$, with positive radius and
centered at $q_i$, such that for any $x\in K$ the distance
function $p\in H_x\cap Q\cap D_{q_i}\to ||p-q_i||\in\mathbb{R}$
has no critical points $p\neq q_i$}} (we already proved a similar
result in \cite{DGF}, Lemma 3.4, (v)). By (\cite{Loj}, pp. 21-28)
it follows that for any $x\in K$ there is a closed ball $C_x
\subseteq \Ps^*$ centered at $x$, for which the induced map
$z\in\pi^{-1}(C_x)\cap D_{q_i}\to\pi(z)\in C_x$ is a Milnor
fibration, with discriminant locus given by $\hc_i\cap{C_x}$.
Since $K$ is compact, we may cover it with finitely many of such
$C_x$'s. So we deduce the existence of a connected closed tubular
neighborhood $\mathcal K$ of $K$ in $\Ps^*$, such that the map:
\begin{equation} \label{nvarphi}
\pi_{\mathcal K}: z\in\pi^{-1}(\mathcal K)\cap D_{q_i}
\to\pi(z)\in \mathcal K
\end{equation}
defines a  $C^{\infty}$-fiber bundle on $\mathcal K
\backslash\hc_i$, and whose fibre $\pi_{\mathcal
K}^{-1}(t)=H_t\cap Q\cap D_{q_i}$, $t\in \mathcal K
\backslash\hc_i$, may be identified with the Milnor fibre.

(ii) Let $\mc_i$ be the local system with fibre $\mc_{i,t}$ at
$t\in \mathcal K\backslash\dc$ given by the image of
$H_{m}(H_t\cap Q\cap D_{q_i}; \mathbb{Q})$ in $H_{m}(H_t\cap Q
;\mathbb{Q})\cong H^{m}(Q_t; \mathbb{Q})$. Notice that, for any
general pencil $L\in \bG(1,\Ps^*)$,  the local system $\mathcal
M_i$  extends, as a local system, $M_i$ on all $L\cap (\mathcal
K\backslash \dc)$ (compare with Notations \ref{naltrenotazioni},
(iii)). In particular  we may assume $M_i=\mc_{i,t}$.
\end{notations}

We are in position to prove Proposition \ref{ntrivial}. We keep
the same notation we introduced before.

\begin{proof}[Proof of Proposition \ref{ntrivial}]

As in (\cite{Steenbrink}, proof of Theorem (2.2)), we need to
consider only the action of $\pi_1(\Ps^*\backslash (\bigcup_{1\leq
j\leq r}\hc_j),t)$.

Consider the finite set $A:=\ell\cap(\bigcup_{j\neq i}\hc_j)$, and
let $a\in A$ be a point. In view of Remark \ref{notesugenfact},
(iii), and Remark \ref{cono}, we may assume that $H_a\cap Q$ has
an isolated singularity at $q_i$. Notice that, a priori, it may
happen that $a\in\ell\cap Q^*$ and so $a\notin K$. But in any
case, since $H_a\cap Q$ has an isolated singularity at $q_i$, as
before, for any $a\in A$ we may construct a closed ball
$D^{(a)}_{q_i}\subseteq {\Ps}^*\times \Ps$, with positive radius
and centered at $q_i$, and a closed ball $C_{a} \subseteq \Ps^*$
centered at $a$, for which the induced map
\begin{equation} \label{nivj}
z\in\pi^{-1}(C_{a})\cap D^{(a)}_{q_i}\to\pi(z)\in C_{a}
\end{equation}
is a Milnor fibration with discriminant locus contained in
$\hc_i\cup Q^*$.  We may assume $D_{q_i}\subseteq D^{(a)}_{q_i}$
for any $a\in A$,  and, shrinking the disks $\Delta_{u}^{\circ}$
($u\in\ell\cap Q^*$) if necessary, we may also assume that the
interior $\mathcal K^{\circ}$ of $\mathcal K$ meets the interior
$C^{\circ}_{a}$ of each $C_{a}$. Therefore, in $(\mathcal
K^{\circ}\cap C^{\circ}_{a})\backslash (\hc_i\cup Q^*)$, the
bundle (\ref{nvarphi}) appears as a subbundle of (\ref{nivj}).

Observe that the image in $H^{m}(Q_t; \mathbb{Q})/T_t$ of the
cohomology of (\ref{nivj}) coincides with $({\mathcal
M_{i,t}+T_t})/{T_t}$ on $(\mathcal K^{\circ}\cap
C^{\circ}_{a})\backslash (\hc_i\cup Q^*)$. This implies that, in a
suitable small analytic neighborhood $\mathcal L$ of $\ell$ in
$\Ps^*$,   the quotient local system $({\mathcal
M_{i,t}+T_t})/{T_t}$ extends on all $\mathcal L\backslash\dc$.
Taking into account Picard-Lefschetz formula, and that the
discriminant locus of (\ref{nivj}) is contained in $\hc_i\cup
Q^*$, we have that $\pi_1(\Ps^*\backslash\dc,t)$ acts trivially on
$({\mathcal M_{i,t}+T_t})/{T_t}$. This holds true for any
$i\in\{1,\cdots,r\}$. Hence, in view of (\ref{ndecomp}) and
(\ref{nvanishing}), it follows that the monodromy action is
trivial on $H^{m}(Q_t;\mathbb{Q})/T_t$. This concludes the proof
of Proposition \ref{ntrivial}.
\end{proof}

By standard classical reasonings as in \cite{La} or \cite{Voisin},
from Theorem \ref{ngeneralfact} we deduce the following:

\begin{corollary}
\label{corollario} $V$ is irreducible.
\end{corollary}

\begin{proof}
Let $\{0\}\neq V'\subset V$ be an invariant subspace. As before,
we may write $H^{m}(Q_t;\mathbb{Q})$ $=$ $U\oplus V'$, for a
suitable invariant subspace $U$. Hence we have $V=(V\cap U)\oplus
V'$. On the other hand,  one knows that $V$ is nondegenerate with
respect to the intersection form $<\cdot,\cdot>$ on $Q_t$
(\cite{PS}, p.167). Therefore,  for some $i\in\{r+1,\dots,s\}$,
there exists $\tau\in (V\cap U)\cup V'$ such that
$<\tau,\delta_i>\neq 0$ ($Span(\delta_i):=V_i)$. From the
Picard-Lefschetz formula it follows that the tangential vanishing
cycle $\delta_i$ lies in $(V\cap U)\cup V'$. If $\delta_i\in V\cap
U$, then by Theorem \ref{ngeneralfact} we deduce $V=V\cap U$
(compare with \cite{La}, \cite{Loj}, \cite{Steenbrink},
\cite{Voisin}), and this is in contrast with the fact that
$\{0\}\neq V'$. Hence $\delta_i\in V'$, and by the same reason
$V'=V$. This proves that  $V$ is irreducible.
\end{proof}

\section{Proof of Theorem \ref{thm2}}

\subsection{The set-up}
\label{setup} Consider the rational map $Y\dasharrow \Ps
:=\Ps(H^0(Y,\ic_{W,Y}(d))^*)$ defined by the linear system
$|H^0(Y,\ic_{W,Y}(d))|$. By \cite{Fulton}, 4.4, such a rational
map defines a morphism $Bl_W(Y)\to \Ps$. We denote by $Q$ the
image of this morphism, i.e.:
\begin{equation}\label{Q}
Q:={Im(Bl_W(Y)\to \Ps)}.
\end{equation}

Set $E :=\Ps(\oc_{Y}(k)\oplus\oc_{Y}(d))$. The surjections
$\oc_{Y}(k)\oplus \oc_{Y}(d) \to \oc_{Y}(d)$ and $\oc_{Y}(k)\oplus
\oc_{Y}(d) \to \oc_{Y}(k)$ give rise to divisors $\Theta \cong
Y\subseteq E$ and $\Gamma\cong Y\subseteq E$, with $\Theta\cap
\Gamma =\emptyset$. The line bundle $\oc_{E}(\Theta )$ is base
point free and the corresponding morphism $E\to
\Ps(H^0(E,\oc_{E}(\Theta
 ))^*)$
sends $E$ to a cone over the Veronese variety of $Y$ (i.e. over
$Y$ embedded via $|H^0(Y,\oc_{Y}(d-k))|$) in such a way that
$\Gamma $ is contracted to the vertex $v_\infty$ and $\Theta $ to
a general hyperplane section. In other words, we may view $E$, via
$E\to \Ps(H^0(E,\oc_{E}(\Theta
 ))^*)$, as the blowing-up of the cone over the Veronese variety
at the vertex, and $\Gamma$ as the exceptional divisor
(\cite{Hartshorne}, p. 374, Example 2.11.4).

From the natural resolution of $\ic_{W,Y}$: $0\to \oc_{Y}(-k-d)
\to \oc_{Y}(-k)\oplus \oc_{Y}(-d) \to \ic_{W,Y} \to 0$, we find
that $Bl_W(Y)=\mathbf{Proj}(\oplus _{i\geq 0}\ic^i_{W,Y})$ is
contained in $E$, and that
$\oc_{E}(\Theta-d\Lambda)\mid_{Bl_W(Y)}\cong\oc_{Bl_W(Y)}(1)$
($\Lambda :=$ pull-back of the hyperplane section of
$Y\subseteq\Ps^N$ through $E\to Y$). Therefore:

(i) we have  natural isomorphisms: $H^0(Y,\ic_{W,Y}(d))\cong
H^0(Y,\oc_{Y}\oplus \oc_{Y}(d-k))\cong H^0(E,\oc_{E}(\Theta))$;

(ii) the linear series $|\Theta |$ cut on $Bl_W(Y)$ the linear
series spanned by the strict transforms $\tX$ of the divisors
$X\in |H^0(Y,\ic_{W,Y}(d))|$, and, sending $E$ to a cone in $\Ps$
over a Veronese variety, restricts to $Bl_W(Y)$ to the map
$Bl_W(Y)\to Q$ defined above. Hence we have  a natural commutative
diagram:
$$
\begin{array}{ccccc}
 Bl_W(Y)&\hookrightarrow  & E  \\
\downarrow &\searrow & &\searrow  \\
Y & \dasharrow &Q&\hookrightarrow & \Ps.\\
\end{array}
$$
By the same reason $\Gamma \cap Bl_W(Y) = \tG$ ($\tG:=$ the strict
transform of $G$ in $Bl_W(Y) $).  Notice that $\tG\cong G$ since
$W$ is a Cartier divisor in  $G$. Similarly $\tX\cong X$ when $G$
is not contained in  $X$;

(iii) since   $|\Theta|$ contracts $\Gamma$ to the vertex
$v_{\infty}$, the map $Bl_W(Y)\to Q$ contracts $\tG $ to
$v_{\infty}\in Q$. Furthermore we have  $Bl_W(Y) \backslash \tG
\cong Q \backslash \{v_{\infty}\}$ and so the hyperplane sections
of $Q$ not containing the vertex are isomorphic, via $Bl_W(Y)\to Q
$, to the corresponding divisors $X\in|H^0(Y,\ic_{W,Y}(d))|$;

(iv) by (ii) above, $\tG $ is a smooth Cartier divisor in
$Bl_W(Y)$, hence $\tG$ is disjoint with $Sing(Bl_W(Y))$. On the
other hand, from (\cite{Vogel}, p. 133, Proposition 4.2.6. and
proof) we know  that $Sing (W)$ is a finite set. The singularities
of $Bl_W(Y)$ must be contained in the inverse image of $Sing(W)$
via $Bl_W(Y)\to Y$: this is a finite set of lines none of which
lying in $Sing (Bl_W(Y))$ because $\tG $ meets all such lines.
Therefore $Sing (Bl_W(Y))$ must be a finite set, and so also
$Sing(Q)$ is. Observe also that $\tG$ is isomorphic to  the
tangent cone to $Q$ at $v_{\infty}$, and its degree is
$k(d-k)^mdeg\,Y$. Hence $Q$ is nonsingular at $v_{\infty}$ only
when $Y=\Ps^{m+1}$, $k=1$ and $d=2$. In this case $X$ is a smooth
quadric, therefore $dim\,{H^m(X;\mathbb Q)_{\perp
W}^{\text{van}}}\leq 1$, and Theorem \ref{thm2} is trivial. So we
may assume $v_{\infty}\in Sing(Q)$.

\subsection{The proof} We are going to prove Theorem \ref{thm2},
that is the irreducibility of the monodromy action on
$H^m(X;\mathbb Q)_{\perp W}^{\text{van}}$. The proof consists in
an application of previous Corollary \ref{corollario} to the
variety $Q\subseteq\Ps$ defined in (\ref{Q}). We keep the same
notation we introduced in \ref{setup}.

\begin{proof}[Proof of Theorem \ref{thm2}]
Consider the variety $Q\subseteq\Ps$ defined in (\ref{Q}). By the
description of it given in \ref{setup}, we know that $Q$ is an
irreducible, reduced, non-degenerate projective variety of
dimension $m+1\geq 2$, with isolated singularities.

Let $L\in \bG(1,\Ps^*)$ be a general pencil of hyperplane sections
of $Q$, and denote by $Q_L$ the blowing-up of $Q$ along the base
locus of $L$, and by $f: Q_L\to L$ the natural map (compare with
Section $3$). Denote by $\{a_1,\dots,a_s\}\subseteq L$ the set of
the critical values of $f$. The fundamental group
$\pi_1(L\backslash\{a_1, \dots ,a_s \}, t)$ ($t=$ general point of
$L$) acts by monodromy on $f^{-1}(t)$, and so on $H^{m}(f^{-1}(t);
\mathbb{Q})$, and this action induces an orthogonal decomposition:
$H^{m}(f^{-1}(t); \mathbb{Q})=I\perp V$, where $I$ is the subspace
of the invariant cocycles, and $V$ is its orthogonal complement.
By Corollary \ref{corollario} we know that $V$ is irreducible.

On the other hand, in view of \ref{setup}, we may identify
$f^{-1}(t)$ with a general $X_t\in |H^0(Y,\ic_{W,Y}(d))|$, and the
action of $\pi_1(L\backslash\{ a_1, \dots ,a_s \}, t)$ with the
action induced on $X_t$ by a general pencil of divisors in
$|H^0(Y,\ic_{W,Y}(d))|$. So, in order to prove Theorem \ref{thm2},
it suffices to prove that $H^m(X_t;\mathbb Q)_{ \perp
W}^{\text{van}}=V$. This  is equivalent  to prove that
$I=H^m(Y;\mathbb Q) +H^m(X_t;\mathbb Q)_{W}^{\text{van}}$. Since
the inclusion $H^m(Y;\mathbb Q) +H^m(X_t;\mathbb
Q)_{W}^{\text{van}}\subseteq I$ is obvious, to prove Theorem
\ref{thm2} it suffices to prove that:
\begin{equation}
\label{interpretazione} I\subseteq H^m(Y;\mathbb Q)
+H^m(X_t;\mathbb Q)_{W}^{\text{van}}.
\end{equation}

To this purpose, let $B_L\subseteq Q$ be the base locus of $L$.
Since $v_{\infty}\notin B_L$, then we may regard $B_L\subseteq
Bl_W(Y)$ via $Bl_W(Y)\to Q$. Notice that $B_L\cong X_t\cap M_L$,
for a suitable general $M_L\in|H^0(Y,\oc_{Y}(d-k))|$. Let
$Bl_W(Y)_L$ be the blowing-up of $Bl_W(Y)$ along $B_L$, and
consider the pencil $f_1:Bl_W(Y)_L\to L$ induced from the natural
map $Bl_W(Y)_L\to Q_L$. We have $Q_L \backslash f^{-1}(\{ a_1,
\dots ,a_s \})\cong Bl_W(Y)_L \backslash f_1^{-1}(\{ a_1, \dots
,a_s \})$. So, if $R_L\to Bl_W(Y)_L$ denotes a desingularization
of $Bl_W(Y)_L$, then the subspace $I$ of the invariant cocycles
can be interpreted via ${R}_L$ as $I=j^*(H^{m} ({R}_L
;\mathbb{Q}))$, where $j$ denotes the inclusion $X_t\subseteq
R_L$.

Denote by $\widetilde W$ and $\widetilde{B_L}$ the inverse images
of $W\subseteq Y$ and $B_L\subseteq Bl_W(Y)$ in $R_L$. The map
$R_L\to Y$ induces an isomorphism
$\alpha_1:R_L\backslash(\widetilde W\cup \widetilde{B_L})\to
Y\backslash (W\cup(X_t\cap M_L))$. Consider the following natural
commutative diagram:
$$
\begin{array}{ccc}
H^m(R_L;\mathbb Q)&\stackrel{\rho_1}{\to}& H^m(R_L\backslash(\widetilde W\cup \widetilde{B_L});\mathbb Q) \\
\stackrel \alpha{}\downarrow &     &  \Vert\stackrel {\alpha_1}{}\\
H^m(Y;\mathbb Q)&\stackrel{\rho_2}{\to}& H^m(Y\backslash (W\cup(X_t\cap M_L));\mathbb Q) \\
\stackrel \beta{}\downarrow &     &  \downarrow\stackrel {\beta_1}{} \\
H^m(X_t;\mathbb Q)&\stackrel{\rho_3}{\to}& H^m(X_t\backslash (W\cup(X_t\cap M_L));\mathbb Q) \\
\end{array}
$$
where $\alpha$ is the Gysin map, and fix $c\in
I=j^*(H^m(R_L;\mathbb Q))$. Let $c'\in H^m(R_L;\mathbb Q)$ such
that $j^*(c')=c$. Since
$\beta_1\circ\alpha_1\circ\rho_1=\rho_3\circ j^*$, then  we have:
$\rho_3(c)= (\rho_3\circ\beta\circ\alpha)(c')$. Hence we have
$c-\beta(\alpha(c'))\in Ker\, \rho_3= Im(H^m(X_t,X_t\backslash
(W\cup(X_t\cap M_L));\mathbb Q)\to H^m(X_t;\mathbb Q))$. Since
$H^m(X_t,X_t\backslash (W\cup(X_t\cap M_L));\mathbb Q)\cong
H_m(W\cup(X_t\cap M_L);\mathbb Q)$ (\cite{Fulton}, (3), p. 371),
we deduce $c-\beta(\alpha(c'))\in Im(H_m(W\cup(X_t\cap
M_L);\mathbb Q)\to H_m(X_t;\mathbb Q)\cong H^m(X_t;\mathbb Q))$.
So to prove  (\ref{interpretazione}), it suffices to prove that
$Im(H_m(W\cup(X_t\cap M_L);\mathbb Q)\to H_m(X_t;\mathbb Q)\cong
H^m(X_t;\mathbb Q))$ is contained in $H^m(Y;\mathbb
Q)+Im(H_m(W;\mathbb Q)\to H_m(X_t;\mathbb Q)\cong H^m(X_t;\mathbb
Q))$.

Since $W$ has only isolated singularities, and  $M_L$ is general,
then $W\cap M_L$ and $X_t\cap M_L$  are smooth complete
intersections. From Lefschetz Hyperplane Theorem and Hard
Lefschetz Theorem it follows that the natural map $H_{m-1}(W \cap
M_L;\mathbb Q)\to H_{m-1}(X_t \cap M_L;\mathbb Q)$ is injective.
Hence, from the Mayer-Vietoris sequence of the pair $(W, X_t\cap
M_L)$ we deduce that the natural map $H_{m}(W;\mathbb Q)\oplus
H_{m}(X_t\cap M_L;\mathbb Q) \to H_{m}(W\cup(X_t\cap M_L);\mathbb
Q)$ is surjective. So to prove (\ref{interpretazione}) it suffices
to prove that $Im(H_m(X_t\cap M_L;\mathbb Q)\to H_m(X_t;\mathbb
Q)\cong H^m(X_t;\mathbb Q))$ is contained in $H^m(Y;\mathbb Q)$.
And this follows from the natural commutative diagram:
$$
\begin{array}{ccc}
H_m(X_t\cap M_L;\mathbb Q)\cong H^{m-2}(X_t\cap M_L;\mathbb
Q)&\stackrel{\rho}{\leftarrow} & H^{m-2}(Y;\mathbb Q)\cong
H_{m+4}(Y;\mathbb Q) \\
\downarrow &     &  \downarrow\stackrel {\cap M_L}{}\\
H_{m}(X_t;\mathbb Q)\cong H^{m}(X_t;\mathbb Q)&
\leftarrow &   H^{m}(Y;\mathbb Q)\cong H_{m+2}(Y;\mathbb Q),\\
\end{array}
$$
taking into account that $\rho$ is an isomorphism by Lefschetz
Hyperplane Theorem. This proves (\ref{interpretazione}), and
concludes the proof of Theorem \ref{thm2}.
\end{proof}

\section{Appendix}
\begin{proof}[Proof of property (\ref{nvanishing})]
First notice that since $f^{-1}(\Delta_i) - D^{\circ}_i\to
\Delta_i$ is a trivial fiber bundle ($D^{\circ}_i$= interior of
$D_i$), then the inclusion $(f^{-1}(a), f^{-1}(a)\cap
D_i)\subseteq (f^{-1}(\Delta_i),f^{-1}(\Delta_i)\cap D_i)$ induces
natural isomorphisms $H_{m}(f^{-1}(a), f^{-1}(a)\cap D_i
;\mathbb{Q})$ $\cong H_{m}(f^{-1}(\Delta_i),f^{-1}(\Delta_i)\cap
D_i ;\mathbb{Q})$ for any $a\in \Delta_i$ (use \cite{Spanier}, p.
200 and 258). So, from the natural commutative diagram:
$$
\begin{array}{ccc}
H_{m}(f^{-1}(a_i+\rho);\mathbb{Q})&\stackrel{\beta}{\to}&H_{m}(f^{-1}(a_i+\rho), f^{-1}(a_i+\rho)\cap D_i ;\mathbb{Q})\\
\stackrel {\alpha}{}\downarrow&  &\Vert\\
H_{m}(f^{-1}(\Delta_i);\mathbb{Q})& {\to}&H_{m}(f^{-1}(\Delta_i), f^{-1}(\Delta_i)\cap D_i ;\mathbb{Q}),\\
\end{array}
$$
we deduce that $Ker\,\alpha\subseteq Ker\,\beta =M_i$.

On the other hand, since  the inclusion $f^{-1}(a_i+\rho)\subseteq
f^{-1}(\Delta_i)$ is the composition of the isomorphism
$f^{-1}(a_i+\rho)\cong g^{-1}(a_i+\rho)$ with
$g^{-1}(a_i+\rho)\subseteq g^{-1}(\Delta_i)$, followed by the
desingularization $g^{-1}(\Delta_i)\to f^{-1}(\Delta_i)$,  we
have: $V_i\subseteq Ker\,\alpha$.
\end{proof}

\medskip

\begin{lemma}\label{nWhitney}
Let $\ell \subseteq\hc_i$ be a general line. For any $u\in\ell\cap
Q^*$, denote by $\Delta_{u}^{\circ}$ an open disk of $\ell$ with
center $u$ and small radius. Consider the compact
$K:=\ell\backslash (\bigcup_{u\in{\ell\cap
Q^*}}\Delta_{u}^{\circ})$. Then there is a closed ball
$D_{q_i}\subseteq{ \Ps}^*\times \Ps$, with positive radius and
centered at $q_i$, such that for any $x\in K$ the distance
function $p\in H_x\cap Q\cap D_{q_i}\to ||p-q_i||\in\mathbb{R}$
has no critical points $p\neq q_i$.
\end{lemma}
\begin{proof}
We argue by contradiction. Suppose the claim is false. Then there
is a sequence of hyperplanes $y_n\in K$, $n\in\mathbb N$,
converging to some $x\in K$, and a sequence of critical points
$p_n\neq q_i$ for the distance function on $H_{y_{n}}\cap Q$,
converging to $q_i$ (we  may assume $p_n$ is smooth for
$H_{y_{n}}\cap Q$). Let $T_{p_{n}, Q}$, $T'_{p_{n},H_{y_{n}}\cap Q
}$ and $s_{q_i,p_n}$ be the corresponding sequences of tangent
spaces and secants, and denote by $r_{q_i,p_n}\subseteq
s_{q_i,p_n}$ the real line meeting $q_i$ and $p_n$. We may  assume
they converge, and we denote by $T$, $T'$, $s$ and $r$ their
limits ($r\subseteq s$). Since $p_n$ is a critical point, then
$r_{q_i,p_n}$ is orthogonal to $T'_{p_{n},H_{y_{n}}\cap Q}$, hence
$r\not \subseteq T'$, and so $T$ is spanned by $T'\cup s$ by
dimension reasons. Since $T'\cup s\subseteq H_x$ then $T\subseteq
H_x$, so $H_x$ contains a limit of tangent spaces of $Q$, with
tangencies converging to $q_i$. This implies that $x\in Q^*$,
contradicting the fact that $x\in K$.
\end{proof}

{\bf{Aknowledgements}}

We would like to thank Ciro Ciliberto   for valuable discussions
and suggestions on the subject of this paper.

\end{document}